\documentclass[11pt]{amsart}
\usepackage{amsmath,amssymb,latexsym,soul,cite,mathrsfs}

\usepackage{color,enumitem,graphicx}
\usepackage[colorlinks=true,urlcolor=blue,
citecolor=red,linkcolor=blue,linktocpage,pdfpagelabels,
bookmarksnumbered,bookmarksopen]{hyperref}
\usepackage[english]{babel}

\usepackage[left=2.9cm,right=2.9cm,top=2.8cm,bottom=2.8cm]{geometry}
\usepackage[hyperpageref]{backref}
\usepackage{ dsfont }
\usepackage[colorinlistoftodos]{todonotes}
\makeatletter
\providecommand\@dotsep{5}
\def\listtodoname{List of Todos}
\def\listoftodos{\@starttoc{tdo}\listtodoname}
\makeatother

\numberwithin{equation}{section}
\def\dis{\displaystyle}

\def\R {{\rm I}\hskip -0.85mm{\rm R}}
\def\N {{\rm I}\hskip -0.85mm{\rm N}}

\def\dis{\displaystyle}

\newtheorem{theorem}{Theorem}[section]
\newtheorem{proposition}[theorem]{Proposition}
\newtheorem{lemma}[theorem]{Lemma}

\title[Multiple solutions for a generalized Schr\"{o}dinger problem]
{Multiple solutions for a generalized Schr\"{o}dinger problem with ``concave-convex'' nonlinearities}

\author[A. V. Santos]{Andrelino V. Santos}
\author[J. R. Santos Jr.]{Jo\~ao R. Santos J\'unior$^\ast$}

\address[A. V. Santos]{\newline\indent Faculdade de Matem\'atica
\newline\indent 
Instituto de Ci\^{e}ncias Exatas e Naturais
\newline\indent 
Universidade Federal do Par\'a
\newline\indent
Avenida Augusto corr\^{e}a 01, 66075-110, Bel\'em, PA, Brazil}
\email{\href{mailto: andrellino77@gmail.com}{andrellino77@gmail.com}}

\address[J. R. Santos Jr.]{\newline\indent Faculdade de Matem\'atica
\newline\indent 
Instituto de Ci\^{e}ncias Exatas e Naturais
\newline\indent 
Universidade Federal do Par\'a
\newline\indent
Avenida Augusto corr\^{e}a 01, 66075-110, Bel\'em, PA, Brazil}
\email{\href{mailto: joaojunior@ufpa.br }{joaojunior@ufpa.br}}

\thanks{$\ast$ Corresponding author: Jo\~ao R. Santos J\'unior was partially supported by CNPq-Proc. 302698/2015-9 and CAPES-Proc. 88881.120045/2016-01, Brazil.}

\begin{document}

\maketitle
\begin{abstract}
A class of generalized Schr\"{o}dinger elliptic problems involving concave-convex and other types of nonlinearities is studied. A reasonable overview about the set of solutions is provided when the parameters involved in the equation assume different real values. 

\end{abstract}
\maketitle

\section{Introduction}


We are interested in investigating the following classes of stationary generalized Schr\"{o}dinger problems
\begin{equation}\label{P}\tag{$P_{\lambda, \mu, q, p}$}
\left \{ \begin{array}{ll}
-div( \vartheta(u)\nabla u)+\frac{1}{2}\vartheta'(u)|\nabla u|^{2}= \lambda |u|^{q-2}u +\mu |u|^{p-2}u& \mbox{in $\Omega$,}\\
u=0 & \mbox{on $\partial\Omega$,}
\end{array}\right.
\end{equation}
where $\Omega\subset\R^{N}$, $N\geq 3$, is a bounded smooth domain, $1<q<4$, $\max\{2, q\}<p<22^{\ast}$, $\lambda$ and $\mu$ are real parameters and $\vartheta:\R\to[1,\infty)$ is a general even $C^{1}$-function whose hypothesis will be posteriorly mentioned.

\medskip

When $\Omega=\R^{N}$, equation \eqref{P} is related to the existence of solitary wave solutions for the parabolic quasilinear Schr\"{o}dinger equation
\begin{equation}\label{evolution}
i\partial_{t}z=-\Delta z+V(x)z-\rho(|z|^{2})z-\Delta(l(|z|^{2}))l'(|z|^{2})z, \ x\in\R^{N},
\end{equation}
where $z:\R\times\R^{N}\to\mathds{C}$, $V:\R^{N}\to\R$ is a given potential and $l, \rho$ are real functions. Equation \eqref{evolution} appears naturally as a model for several physical phenomena, depending on the type of function $l$ considered. In fact, if $l(s)=s$, \eqref{evolution} describes the behavior of a superfluid film in plasma physics, see \cite{Kur}. For $l(s)=(1+s)^{1/2}$, \eqref{evolution} models the self-channeling of a high-power ultrashort laser in matter, see \cite{BG, BMMLB, CS, LSS}. Furthermore, \eqref{evolution} also appears in plasma physics and fluid mechanics \cite{LS}, in dissipative quantum mechanics \cite{Has}, in the theory of Heisenberg ferromagnetism and magnons \cite{QC} and in condensed matter theory \cite{MF}.

\medskip

If we take $z(t, x)=e^{-iEt}u(x)$ in \eqref{evolution}, we get the corresponding steady state equation
\begin{equation}\label{schro}
-\Delta u+V(x)u-\Delta(l(u^{2}))l'(u^{2})u= \rho(u) \ \mbox{in $\R^{N}$}.
\end{equation}
In the case that $\rho(s)=\lambda |s|^{q-2}s +\mu |s|^{p-2}s$ and $\R^{N}$ is replaced by $\Omega$, problem \eqref{schro} can be obtained from \eqref{P}, simply by choosing $\vartheta(s)=1+(l(s^{2})')^{2}/2$, for some $C^{2}$-function $l$. 

\medskip

Many authors have studied stationary Schr\"{o}dinger problems like \eqref{P} under different nonlinearities and functions $\vartheta$, when $\Omega=\R^{N}$. Without any intention to provide a complete overview about the matter, we just refer the reader to some seminal contributions: In the case $\vartheta(s)=1+2s^{2}$, 
see \cite{CJ, DMS, DMO, DPY, LWW, PSW, SV, YWA}. In the case $\vartheta(s)=1+s^{2}/2(1+s^{2})$, see \cite{DHS, ShW1, ShW2}.

\medskip

The main goal of the present paper is provide a reasonable outline about the existence of multiple solutions for problem \eqref{P}, when the parameters involved assume different values and function $\vartheta$ satisfies general conditions which cover some of the cases previously mentioned. More specifically, we are assuming that:
\begin{enumerate}
\item[\underline{$(\vartheta_1)$}] $s\mapsto \vartheta(s)$ is decreasing in $(-\infty, 0)$ and increasing in $(0, \infty)$;
\item[\underline{$(\vartheta_2)$}] $s\mapsto \vartheta(s)/s^{2}$ nondecreasing in $(-\infty, 0)$ and nonincreasing in $(0, \infty)$;
\item[\underline{$(\vartheta_3)$}] $\lim_{|s|\to \infty}\vartheta(s)/s^{2}=\alpha^{2}/2$, for some $\alpha>0$.
\end{enumerate}

Some examples of functions satisfying $(\vartheta_{1})-(\vartheta_{3})$ can be given by: 
$$
\vartheta_{\ast}(s)=1+2s^{2}, \ \vartheta_{\#}(s)=1+\frac{s^{2}}{2(1+s^{2})}+s^{2} \ \mbox{and} \ \vartheta_{\dagger}(s)=1+\ln(1+e^{s^{2}}),
$$
other examples can be found in \cite{SSS}, where the authors consider the problem \eqref{P} with power type nonlinearities.

\medskip

Due to the nature of the generalized Schr\"{o}dinger operator, some interesting phenomena can be observed when one compares \eqref{P} to the classical concave-convex problem involving the laplacian operator. For example, results of existence of infinitely many solutions with ``high energy'', commonly influenced by convex part of the nonlinearity, are just occurring when $p>4$. Moreover, multiplicity of solutions with ``low energy'' has been obtained for values of $q$ that are not in the interval  $(1, 2)$. More specifically, what is noticed is the existence of a ``grey zone'', namely, $2\leq q<p\leq 4$, where the set of solutions has an intermediate behaviour, presenting simultaneously influence of both powers as well as of the length of $\lambda$ and $\mu$, see Theorem \ref{teor3}. Our main results are as follows:

\begin{theorem}\label{teor1}
The following claims hold:
\begin{enumerate}
\item[$(i)$] If $\lambda, \mu\leq 0$, then \eqref{P} does not have any nontrivial solution;

\item[$(ii)$] Suppose that $\vartheta$ satisfies $(\vartheta_{1})-(\vartheta_{2})$, $1<q\leq 2$ and $p\geq 4$ hold. If $\lambda<0$, then \eqref{P} does not have solutions $u$ satisfying $J_{\lambda, \mu}(f^{-1}(u))\leq 0$. Analogously, if $\mu<0$, then \eqref{P} does not have solutions $u$ satisfying $J_{\lambda, \mu}(f^{-1}(u))\geq 0$;

\item[$(iii)$] Suppose that $\vartheta$ satisfies $(\vartheta_{1})-(\vartheta_{3})$. If $\max\{2, q\}< p\leq 4$ and $\lambda<0$, then there exists $\mu_{\ast}>0$ such that \eqref{P} does not have any nontrivial solution, whatever $\mu\in (0, \mu_{\ast})$. Moreover, if $1<q<2<p\leq 4$ and $\lambda>0$, then there exists $s_{\ast}>0$ such that \eqref{P} does not have solutions $u$ satisfying $J_{\lambda, \mu}(f^{-1}(u))\geq 0$, whatever $\mu\in (-s_{\ast}, s_{\ast})$.

\item[$(iv)$] Suppose that $\vartheta$ satisfies $(\vartheta_{1})-(\vartheta_{3})$. If $2\leq q< 4$ and $\mu<0$, then there exists $\lambda_{\ast}>0$ such that \eqref{P} does not have any nontrivial solution, whatever $\lambda\in (0, \lambda_{\ast})$. Moreover, if $2\leq q<p\leq 4$ and $\mu>0$, then there exists $t_{\ast}>0$ such that \eqref{P} does not have solutions $u$ satisfying $J_{\lambda, \mu}(f^{-1}(u))\leq 0$, whatever $\lambda\in (-t_{\ast}, t_{\ast})$.

\item[$(v)$] Suppose that $\vartheta$ satisfies $(\vartheta_{1})-(\vartheta_{3})$. If $2\leq q<p\leq 4$, then there exist $r_{\ast}>0$ such that \eqref{P} does not have any nontrivial solution, whatever $\lambda, \mu\in (-r_{\ast}, r_{\ast})$.
\end{enumerate}
\end{theorem}


\begin{theorem}\label{teor3}
Suppose that $\vartheta$ satisfies $(\vartheta_{1})-(\vartheta_{3})$. The following claims hold:
\begin{enumerate}
\item[$(i)$] Let $\lambda\in\R$, $\mu>0$ and $1<q<4$. If $4<p<22^{\ast}$, then \eqref{P} has a sequence of solutions $\{u_{n}\}$ with $J_{\lambda, \mu}(f^{-1}(u_{n}))\to\infty$. Furthermore, if $\max\{q, 2\}<p< 4$, then for each $k\in\N$ there exists $\mu_{k}>0$ such that $\eqref{P}$ has at least $k$ pairs of nontrivial solutions $u_{k}$ with $J_{\lambda, \mu}(f^{-1}(u_{k}))> 0$, provided that $\mu\in (\mu_{k}, \infty)$;

\item[$(ii)$] Let $\lambda>0$, $\mu\in\R$ and $p\neq 4$. If $1<q<2$, then \eqref{P} has a sequence of solutions $\{u_{n}\}$ with $J_{\lambda, \mu}(f^{-1}(u_{n}))<0$ and $J_{\lambda, \mu}(f^{-1}(u_{n}))\to 0$. Furthermore, if $2\leq q<4$, then for each $k\in\N$ there exists $\lambda_{k}>0$ such that $\eqref{P}$ has at least $k$ pairs of nontrivial solutions $u_{k}$ with $J_{\lambda, \mu}(f^{-1}(u_{k}))< 0$, provided that $\lambda\in (\lambda_{k}, \infty)$.

\item[$(iii)$] Let $\lambda>0$, $\mu<\lambda_{1}\alpha^{2}/4$ and $p=4$. Then, for each $k\in\N$ there exists $\lambda_{k}>0$ such that $\eqref{P}$ has at least $k$ pairs of nontrivial solutions $u_{k}$ with $J_{\lambda, \mu}(f^{-1}(u_{k}))< 0$, provided that $\lambda\in (\lambda_{k}, \infty)$, where $\alpha$ is defined in $(\vartheta_{3})$.

\end{enumerate}
\end{theorem}

Throughout the paper $|A|$ denotes the Lebesgue measure of a measurable set $A\subset\R^{N}$, $[1<u]:=\{x\in\Omega: 1<u(x)\}$, $\lambda_{1}$ is the first eigenvalue of laplacian operator with homogeneous Dirichlet boundary condition and $C, C_{0}, C_{1}, C_{2}$ stand for positive constants whose exact value is not relevant for our purpose.



The paper is organized as follows. 

In Section \ref{se:eigen} we study a suitable change of variable which becomes problem \eqref{P} in a more manageable one. 
In Section \ref{se:trivial} we prove nonexistence results.
In Section \ref{sec:ult} we prove existence results.


\section{Preliminaries} \label{se:eigen}

Our approach consists in switching the task to look for solutions of the general semilinear problem
\begin{equation}\label{P}\tag{$P_{\lambda, \mu, q, p}$}
\left \{ \begin{array}{ll}
-div( \vartheta(u)\nabla u)+\frac{1}{2}\vartheta'(u)|\nabla u|^{2}=\lambda |u|^{q-2}u +\mu |u|^{p-2}u & \mbox{in $\Omega$,}\\
u=0 & \mbox{on $\partial\Omega$,}
\end{array}\right.
\end{equation}
by task to find solutions of 
\begin{equation}\label{P'}\tag{$P_{\lambda, \mu, q, p}'$}
\left \{ \begin{array}{ll}
-\Delta v= \lambda f'(v)|f(v)|^{q-2}f(v)+\mu f'(v)|f(v)|^{p-2}f(v) & \mbox{in $\Omega$,}\\
v=0 & \mbox{on $\partial\Omega$,}
\end{array}\right.
\end{equation}
where $f\in C^{2}(\R)$ is a solution of the ordinary differential equation
\begin{equation}\label{ODE}\tag{ODE}
f'(s)=\frac{1}{\vartheta(f(s))^{1/2}} \ \mbox{for} \ s>0 \ \mbox{and} \ f(0)=0,
\end{equation}
with $f(s)=-f(-s)$ for $s\in (-\infty, 0)$. Since $f$ is odd and $\vartheta$ is even, equation \eqref{ODE} is yet true for negative values. It is well known that $v$ is a weak solution of \eqref{P'} if, and only if, $u=f(v)$ is a weak solution of \eqref{P}, see \cite{SSS} or \cite{ShW1}.

\medskip

Despite the proof of next lemma can also be found in \cite{SSS}, for the reader's convenience and by its relevant role throughout the paper, we provide it here.

\begin{lemma} \label{prop1}
Let $\vartheta\in C^{1}(\R)$ and $f$ a solution of \eqref{ODE}. The following claims hold:
\begin{enumerate}
\item[$(i)$] $f$ is uniquely defined and it is an increasing $C^{2}$-diffeomorphism, with $f''(s)=-\vartheta'(f(s))/2\vartheta(f(s))^{2}$, for all $s>0$;
\item[$(ii)$] $0< f'(s)\leq 1$, for all $s\in\R$;
\item[$(iii)$]  $\lim_{s\to 0}f(s)/s=1/\vartheta(0)^{1/2}$; 
\item[$(iv)$]  $ |f(s)|\leq |s|$, for all $s\in\R$; 
\item[$(v)$] Suppose $(\vartheta_{1})-(\vartheta_{2})$ hold. Then, $|f(s)|/2\leq f'(s)|s|<|f(s)|$, for all $s\in\R\backslash\{0\}$,
and the map $s\mapsto |f(s)|/\sqrt{|s|}$ is nonincreasing in $(-\infty, 0)$ and nondecreasing in $(0, \infty)$; 
\item[$(vi)$] Suppose that $(\vartheta_{1})-(\vartheta_{3})$ hold. Then,
$$
\dis\lim_{|s|\to\infty}\frac{|f(s)|}{\sqrt{|s|}}=\left(\frac{8}{\alpha^{2}}\right)^{1/4} \ \mbox{and} \ \lim_{|s|\to\infty}\frac{f(s)}{s}=0,
$$
where $\alpha$ is given in $(\vartheta_{3})$. 
\end{enumerate}
\end{lemma}

\begin{proof}
$(i)$-$(ii)$ Existence, uniqueness, regularity, monotonicity and $(ii)$ follow directly from \eqref{ODE}.  To see that $f(\R)=\R$, observe that $f(s)=(\Upsilon^{-1})(s)$, where
$$
\Upsilon(t)=\int_{0}^{t}\vartheta(r)^{1/2}dr.
$$
Since $\vartheta\geq 1$, $|\Upsilon(t)|\geq |t|$ for all $t\in\R$. Consequently, $\lim_{|t|\to\infty}|\Upsilon(t)|=\infty$. Thence, $\lim_{|s|\to\infty}|f(s)|=\infty$.

$(iii)$ Notice that, by L'H\^{o}spital rule, we get
$$
\dis\lim_{s\to 0}\frac{f(s)}{s}=\lim_{s\to 0}f'(s)=\frac{1}{\vartheta(0)^{1/2}}.
$$

$(iv)$ It follows from $(ii)$. $(v)$ Since $f$ is odd and $\vartheta$ is even, it is sufficient to prove the inequalities for $s>0$. For that, let $r_{1}:[0,\infty)\to\R$ defined by
\begin{equation*}
r_{1}(s)=f(s)\vartheta(f(s))^{1/2}-s.
\end{equation*}
Notice that $r_{1}(0)=0$ and, by \eqref{ODE} and $(\vartheta_{1})$, we have
$$
r_{1}'(s)=\vartheta'(f(s))f(s)/2\vartheta(f(s))>0.
$$
Therefore, the second inequality in $(v)$ follows. Now, to prove the first inequality in $(v)$, let $r_{2}:[0,\infty)\to\R$ be defined by
\begin{equation*}
r_{2}(s)=2s-f(s)\vartheta(f(s))^{1/2}.
\end{equation*}
We have that $r_{2}(0)=0$ and, by \eqref{ODE} and $(\vartheta_{2})$,
$$
r_{2}'(s)=1-\vartheta'(f(s))f(s)/2\vartheta(f(s))\geq 0,
$$
showing that the inequality in $(v)$ holds. Moreover, since
$$
\left(\frac{f(s)}{\sqrt{s}}\right)'=\frac{2f'(s)s-f(s)}{2s\sqrt{s}}\geq 0, \ \forall \ s>0,
$$
the second part of $(v)$ follows.

$(vi)$ Observe that from $(v)$, we have 
$$
\lim_{|s|\to \infty}\frac{|f(s)|}{\sqrt{|s|}}=l, \ \mbox{with $l\in (0,\infty]$}.
$$ 
Again, since $f$ is odd and $\vartheta$ is even, it is sufficient to consider the case $s\to\infty$. Suppose that 
\begin{equation}\label{52}
\lim_{s\to \infty}f(s)/\sqrt{s}=\infty.
\end{equation} 
If this is the case then, by $(i)$, we get $f(s)\to\infty$ as $s\to\infty$. By applying the L'H\^{o}spital rule and using $(\vartheta_{3})$, we conclude from \eqref{52}, that
\begin{eqnarray*}
\lim_{s\to\infty}\frac{f(s)}{\sqrt{s}}&=&\dis\lim_{s\to\infty}2f'(s)\sqrt{s}\\
&=&2\dis\lim_{s\to\infty}\sqrt{\frac{s}{\vartheta(f(s))}}\\
&=&2\sqrt{\frac{\dis\lim_{s\to\infty}\left(\sqrt{s}/f(s)\right)^{2}}{\dis\lim_{s\to\infty}\vartheta(f(s))/f(s)^{2}}}\\
&=&2\sqrt{\frac{0}{(\alpha^{2}/2)}}=0.
\end{eqnarray*}
Showing that 
\begin{equation}\label{51}
\lim_{s\to\infty}f(s)/\sqrt{s}=0.
\end{equation}
Since \eqref{51} contradicts \eqref{52}, it follows that $0<\lim_{s\to\infty}f(s)/\sqrt{s}=l<\infty$. Applying one more time the L'H\^{o}spital rule, we have
$$
l=2\sqrt{\frac{\dis\lim_{s\to\infty}\left(\sqrt{s}/f(s)\right)^{2}}{\dis\lim_{s\to\infty}\vartheta(f(s))/f(s)^{2}}}=2\sqrt{\frac{1/l^{2}}{(\alpha^{2}/2)}}.
$$
Or equivalently,
\begin{equation}\label{53}
l=\left(\frac{8}{\alpha^{2}}\right)^{1/4}. 
\end{equation}
On the other hand, from \eqref{53},
$$
\dis\lim_{s\to\infty}\frac{f(s)}{s}=\dis\lim_{s\to\infty}\frac{f(s)}{\sqrt{s}}\frac{1}{\sqrt{s}}=\left(\frac{8}{\alpha^{2}}\right)^{1/4}\times 0=0. 
$$  
\end{proof}

%

Naturally, a weak solution of \eqref{P'} is a function $u\in H_{0}^{1}(\Omega)$ satisfying
\begin{equation}\label{ws}
\int_{\Omega}\nabla u\nabla v dx=\lambda\int_{\Omega}f'(u)\left|f(u)\right|^{q-2}f(u)v dx+\mu\int_{\Omega}f'(u)\left|f(u)\right|^{p-2}f(u)v dx,
\end{equation}
for all $v\in H_{0}^{1}(\Omega)$. Moreover, the energy functional $J_{\lambda, \mu}: H_{0}^{1}(\Omega)\to\R$ associated to \eqref{P'} is 
\begin{equation}\label{ef}
J_{\lambda, \mu}(u)=\frac{1}{2}\|u\|^{2} -\frac{\lambda}{q}\int_{\Omega}\left|f(u)\right|^{q} dx-\frac{\mu}{p}\int_{\Omega}\left|f(u)\right|^{p} dx.
\end{equation}
Lemma \ref{prop1} assures that the previous notion of weak solution makes sense, as well as ensures that functional $J_{\lambda, \mu}$ is well defined and is $C^{1}$.
%
%
%
Before finishing this section, we are going to introduce two technical lemmas which will be very helpful later on.

\begin{lemma}\label{lema1}

Let $\{u_{n}\}$ be a sequence of measurable functions $u_{n}:\Omega\to \R$. Then, 
$$
\chi_{\left[1<\liminf\limits_{n\to\infty} u_{n}\right]}(x)\leq \liminf\limits_{n\to\infty}\chi_{[1<u_{n}]}(x) \ \mbox{in $\Omega$}.
$$


\end{lemma}

\begin{proof}
Let us define $u:=\liminf\limits_{n\to\infty} u_{n}$ and $g:\Omega\to\{0, 1\}$ by 
$$
g(x)=\liminf\limits_{n\to\infty}\chi_{[1< u_{n}]}(x).
$$
If $g\equiv 1$, there is nothing to be proven. Otherwise, it is sufficient to prove that if $g(x)=0$, then $\chi_{[1< u]}(x)=0$. Indeed, observe that if $g(x)=0$ then there exists a subsequence $u_{n_{k}}$ where $\{n_{k}\}\subset\N$ depends on $x$, such that
$$
\chi_{[1< u_{n_{k}}]}(x)=0, \ \forall \ k\in\N.
$$
Equivalently,
$$
u_{n_{k}}(x)\leq 1, \ \forall \ k\in\N.
$$
Passing to the lower limit as $k$ goes to infinity, we obtain
$$
u(x)= \liminf\limits_{n\to\infty} u_{n}(x)\leq \liminf\limits_{k\to\infty} u_{n_k}(x)\leq 1,
$$
or yet
$$
\chi_{[1< u]}(x)=0.
$$
\end{proof}

Now on, let us agree that, $\{e_{j}\}$ stands for a Hilbertian basis of $H_{0}^{1}(\Omega)$ composed by functions in $L^{\infty}(\Omega)$ (for example the basis composed by eigenfunctions of laplacian operator with Dirichlet boundary condition),
$$
X_{j}:=Span\{e_{j}\}, Y_{k}:=\dis\oplus_{j=0}^{k}X_{j} \ \mbox{and} \ Z_{k}:=\overline{\dis\oplus_{j=k}^{\infty}X_{j}}.
$$

Since $|f(s)|$ behaves like $|s|$ near the origin and like $|s|^{1/2}$ at infinity, next lemma will be very helpful to get some important estimates for the existence results.

\begin{lemma}\label{lema2}

Let $S_{k}$ be the unit sphere of $Y_{k}$. There exist positive constants $\beta_{k}, \beta_{k}(r), \alpha_{k}, \tau_{k}$ such that:

\begin{enumerate}
\item[$(i)$]   
\begin{equation}\label{ineq51}
\beta_{k}\leq |[1<|su|]|,
\end{equation}
for all $u\in S_{k}$ and $s>\alpha_{k}$, and 
\begin{equation}\label{segundo}
[|su|<1]=\Omega,
\end{equation}
for all $u\in S_{k}$ and $0<s<\tau_{k}$. 

\medskip

\item[$(ii)$] for each $r\in [1, 2^{\ast}]$, 
\begin{equation}\label{terceiro}
\beta_{k}(r)\leq \int_{[1<|su|]}|u|^{r}dx, 
\end{equation}
for all $u\in S_{k}$ and $s>\alpha_{k}$. 
\end{enumerate}

\end{lemma}

\begin{proof}

$(i)$ First, we are going to prove that \eqref{ineq51} holds.
Indeed, suppose that there exist $\{s_{n}\}\subset (0, \infty)$ and $\{u_{n}\}\subset S_{k}$ with $s_{n}\to\infty$ and
\begin{equation}\label{equ1}
|[1<|s_{n}u_{n}|]|\to 0 \ \mbox{as} \ n\to\infty.
\end{equation}
Since $Y_{k}$ has finite dimension, there exists 
\begin{equation}\label{equ2}
u\in S_{k} 
\end{equation}
such that, up to a subsequence, $u_{n}\to u$ in $H_{0}^{1}(\Omega)$ and
$$
u_{n}(x)\to u(x) \ \mbox{a.e. in $\Omega$}.
$$
Therefore,
\begin{equation}\label{equ3}
|s_{n}u_{n}|\to \infty \ \mbox{in $[u\neq 0]$}.
\end{equation}
It follows from \eqref{equ2}, \eqref{equ3}, Lemma \ref{lema1}$(i)$, Fatou Lemma and \eqref{equ1} that
\begin{eqnarray*}
0<|[u\neq 0]|&\leq& |[1<\liminf\limits_{n\to\infty}|s_{n}u_{n}|]|\\
&=&\int_{\Omega}\chi_{[1<\liminf\limits_{n\to\infty}|s_{n}u_{n}|]}(x) dx\\
&\leq & \int_{\Omega}\liminf\limits_{n\to\infty}\chi_{[1<|s_{n}u_{n}|]}(x) dx\\
&\leq & \liminf\limits_{n\to\infty} \int_{\Omega}\chi_{[1<|s_{n}u_{n}|]}(x) dx\\
&=& \liminf\limits_{n\to\infty} |[1<|s_{n}u_{n}|]|=0.
\end{eqnarray*}
A clear contradiction. Therefore \eqref{ineq51} holds. Now, in order to prove \eqref{segundo}, observe that if $u\in S_{k}$ then, by Cauchy-Schwarz inequality
\begin{equation}\label{elle}
|u(x)|=\left|\sum_{j=0}^{k}y_{j}e_{j}(x)\right|\leq \left(\sum_{j=0}^{k}y_{j}^{2}\right)\left(\sum_{j=0}^{k}e_{j}(x)^{2}\right)\leq (k+1)M^{2},
\end{equation}
where $M:=\max_{j=0}^{k}|e_{j}|_{\infty}$. Consequently, choosing $\tau_{k}:=1/(k+1)M^{2}$ the result follows.

$(ii)$ By Fatou Lemma, Lemma \ref{lema1} and since $Y_{k}$ has finite dimension, we have
\begin{eqnarray*}
\liminf\limits_{s\to\infty}\int_{[1<|su|]}|u|^{r}dx&=&\liminf\limits_{s\to\infty}\int_{\Omega}|u|^{r}\chi_{[1<|su|]}(x)dx\\
&\geq& \int_{\Omega}|u|^{r}\liminf\limits_{s\to\infty}\chi_{[1<|su|]}(x)dx\\
&\geq & \int_{\Omega}|u|^{r}\chi_{[u\neq 0]}(x)dx\\
&=&  \int_{\Omega}|u|^{r}dx\geq\zeta_{k}(r),
\end{eqnarray*}
for all $u\in S_{k}$ and some $\zeta_{k}(r)>0$. Choosing $0<\beta_{k}(r)<\zeta_{k}(r)$, the result is proven.
 

\end{proof}


\section{Nonexistence results}\label{se:trivial}

{\bf Proof of Theorem \ref{teor1}:}\\ 

$(i)$ Indeed, by $f(0)=0$ and Lemma \ref{prop1}$(ii)$ we have $f(s)s\geq 0$ for all $s\in\R$. Thus, if $u$ is a solution, then
$$
\|u\|^{2}= \lambda\int_{\Omega}f'(u)\left|f(u)\right|^{q-2}f(u)udx+\mu\int_{\Omega}f'(u)\left|f(u)\right|^{p-2}f(u)u dx\leq 0.
$$
Therefore $u=0$. 

\medskip

$(ii)$ Suppose that $\lambda<0$ and $u$ is a nontrivial weak solution of \eqref{P'}. By previous item, we have $\mu>0$. By Lemma \ref{prop1}$(v)$,
\begin{equation}\label{vale}
\lambda\int_{\Omega}\left|f(u)\right|^{q}dx+\frac{\mu}{2}\int_{\Omega}\left|f(u)\right|^{p} dx< \|u\|^{2}.
\end{equation}
If $J_{\lambda, \mu}(u)\leq 0$, then
$$
\frac{1}{2}\|u\|^{2}-\frac{\lambda}{q}\int_{\Omega}\left|f(u)\right|^{q}dx-\frac{\mu}{p}\int_{\Omega}\left|f(u)\right|^{p} dx\leq 0.
$$
Thus,
\begin{equation}\label{venga}
\|u\|^{2}\leq \frac{2\lambda}{q}\int_{\Omega}\left|f(u)\right|^{q}dx+\frac{2\mu}{p}\int_{\Omega}\left|f(u)\right|^{p} dx.
\end{equation}
By comparing \eqref{vale} and \eqref{venga}, we get
$$
0\leq \lambda\left(1-\frac{2}{q}\right)\int_{\Omega}\left|f(u)\right|^{q}dx+\mu\left(\frac{1}{2}-\frac{2}{p}\right)\int_{\Omega}\left|f(u)\right|^{p}dx< 0,
$$
whenever $1<q\leq 2$ and $p\geq 4$. A clear contradiction.

\medskip

Now, let $\mu<0$ and $u$ be a weak solution of \eqref{P'}. Again, by item $(i)$, we have $\lambda>0$. By Lemma \ref{prop1}$(v)$,
\begin{equation}\label{valeya}
\|u\|^{2}< \lambda\int_{\Omega}\left|f(u)\right|^{q}dx+\frac{\mu}{2}\int_{\Omega}\left|f(u)\right|^{p} dx.
\end{equation}
If $J_{\lambda, \mu}(u)\geq 0$, then
$$
\frac{1}{2}\|u\|^{2}-\frac{\lambda}{q}\int_{\Omega}\left|f(u)\right|^{q}dx-\frac{\mu}{p}\int_{\Omega}\left|f(u)\right|^{p} dx\geq 0.
$$
Thus,
\begin{equation}\label{vengaya}
\frac{2\lambda}{q}\int_{\Omega}\left|f(u)\right|^{q}dx+\frac{2\mu}{p}\int_{\Omega}\left|f(u)\right|^{p} dx\leq\|u\|^{2}.
\end{equation}
Comparing \eqref{valeya} and \eqref{vengaya}, we get
$$
0< \lambda\left(1-\frac{2}{q}\right)\int_{\Omega}\left|f(u)\right|^{q}dx+\mu\left(\frac{1}{2}-\frac{2}{p}\right)\int_{\Omega}\left|f(u)\right|^{p}dx\leq 0.
$$
for all $1<q\leq 2$ and $p\geq 4$. The result follows.

\medskip

$(iii)$ If $\max\{2, q\}<p\leq 4$, $\lambda<0$ and $u$ is a nontrivial weak solution of \eqref{P'}, then, by $f(0)=0$ and Lemma \ref{prop1}$(ii)$, $f(s)s\geq 0$ for all $s\in\R$. Moreover, by item $(i)$, we have $\mu>0$. Thence,
$$
\|u\|^{2}\leq \mu\int_{\Omega}f'(u)\left|f(u)\right|^{p-1}|u| dx.
$$
By Lemma \ref{prop1}$(v)$, 
\begin{equation}\label{conti}
\|u\|^{2}\leq \mu\int_{\Omega}\left|f(u)\right|^{p} dx.
\end{equation}
It follows from items $(v)$ and $(vi)$ of Lemma \ref{prop1} that
$$
|f(s)|\leq (8/\alpha^{2})^{1/4}|s|^{1/2},
$$
for all $|s|>1$. Thus, by Lemma \ref{prop1}$(iv)$ and since $2\leq p\leq 4$,
\begin{eqnarray}\nonumber\label{finit}
\int_{\Omega}\left|f(u)\right|^{p} dx&\leq& \int_{[|u|\leq 1]}|u|^{p}dx+ (8/\alpha^{2})^{p/4}\int_{[|u|>1]}\left|u\right|^{p/2} dx\\ \nonumber
&\leq & \int_{[|u|\leq 1]}|u|^{2}dx+ (8/\alpha^{2})^{p/4}\int_{[|u|>1]}\left|u\right|^{2} dx\\
&\leq & [1+(8/\alpha^{2})^{p/4}]\int_{\Omega}\left|u\right|^{2} dx.
\end{eqnarray}
By \eqref{conti}, \eqref{finit} and Sobolev embeddings,
\begin{equation}\label{defi}
\|u\|^{2}\leq \mu[1+(8/\alpha^{2})^{p/4}]|u|_{2}^{2}\leq \mu[1+(8/\alpha^{2})^{p/4}]C_{1}\|u\|^{2}.
\end{equation}
Since $u$ is a nontrivial solution, we obtain 
\begin{equation}\label{navi}
0<\frac{1}{[1+(8/\alpha^{2})^{p/4}]C_{1}}=:\mu_{\ast}\leq\mu.
\end{equation}
To prove the second part, suppose that $\lambda>0$ and $u$ is a nontrivial solution with $J_{\lambda, \mu}(u)\geq 0$. It follows from Lemma \ref{prop1}$(v)$ that
\begin{eqnarray*}
\|u\|^{2}&\leq& \lambda\int_{\Omega}\left|f(u)\right|^{q} dx+|\mu|\int_{\Omega}\left|f(u)\right|^{p} dx\\
&\leq& \frac{q}{2}\|u\|^{2}+|\mu|\left(1+\frac{q}{p}\right)\left|f(u)\right|^{p} dx
\end{eqnarray*}
Consequently,
$$
\left(1-\frac{q}{2}\right)\|u\|^{2}\leq |\mu|\left(1+\frac{q}{p}\right)\left|f(u)\right|^{p} dx.
$$
As $2\leq p\leq 4$, by \eqref{finit},
$$
\left(1-\frac{q}{2}\right)\|u\|^{2}\leq |\mu|\left(1+\frac{q}{p}\right)[1+(8/\alpha^{2})^{p/4}]C_{1}\|u\|^{2}.
$$
Since $1<q<2$, we have
$$
0<\frac{\left(1-\frac{q}{2}\right)}{\left(1+\frac{q}{p}\right)[1+(8/\alpha^{2})^{p/4}]C_{1}}\leq |\mu|.
$$
The result is proven. 

\medskip

$(iv)$ Let $2\leq q<4$, $\mu<0$ and $u$ be a nontrivial weak solution of \eqref{P'}, by Lemma \ref{prop1}$(v)$
$$
\|u\|^{2}\leq \lambda\int_{\Omega}\left|f(u)\right|^{q} dx.
$$
By item $(i)$, \eqref{finit} and Sobolev embeddings,
\begin{equation}\label{defi}
\|u\|^{2}\leq \lambda[1+(8/\alpha^{2})^{q/4}]C_{1}\|u\|^{2}.
\end{equation}
Since $u$ is a nontrivial solution, we obtain 
\begin{equation}\label{navi}
0<\frac{1}{[1+(8/\alpha^{2})^{q/4}]C_{1}}=:\lambda_{\ast}\leq\lambda.
\end{equation}

\medskip

Finally, suppose that $\mu>0$ and $u$ is a nontrivial solution with $J_{\lambda, \mu}(u)\leq 0$. It follows from Lemma \ref{prop1}$(v)$ that
\begin{eqnarray*}
\|u\|^{2}&\geq& -|\lambda|\int_{\Omega}\left|f(u)\right|^{q} dx+\frac{\mu}{2}\int_{\Omega}\left|f(u)\right|^{p} dx\\
&\geq& \frac{p}{4}\|u\|^{2}-|\lambda|\left(1+\frac{p}{2q}\right)\int_{\Omega}\left|f(u)\right|^{q} dx.
\end{eqnarray*}
Since $p<4$,
$$
0<\left(1-\frac{p}{4}\right)\|u\|^{2}\leq |\lambda|\left(1+\frac{p}{2q}\right)\int_{\Omega}\left|f(u)\right|^{q} dx.
$$
Since $2\leq q<4$, by \eqref{finit}
$$
\left(1-\frac{p}{4}\right)\|u\|^{2}\leq |\lambda|\left(1+\frac{p}{2q}\right)[1+(8/\alpha^{2})^{q/4}]C_{1}\|u\|^{2}.
$$
Therefore
$$
0<\frac{\left(1-\frac{p}{4}\right)}{\left(1+\frac{p}{2q}\right)[1+(8/\alpha^{2})^{q/4}]C_{1}}\leq |\lambda|.
$$

$(v)$ Let $2\leq q<p\leq 4$ and $u$ be a nontrivial weak solution of \eqref{P'}. By Lemma \ref{prop1}$(v)$ and \eqref{finit},
$$
\|u\|^{2}\leq |\lambda|\int_{\Omega}|f(u)|^{q}dx+|\mu|\int_{\Omega}|f(u)|^{p}dx\leq \left[|\lambda|[1+(8/\alpha^{2})^{p/4}]C_{1}+|\mu|[1+(8/\alpha^{2})^{p/4}]C_{2}\right]\|u\|^{2}.
$$
Since $u$ is nontrivial, the result follows.
$\square$


\section{Multiplicity of solutions}\label{sec:ult}

The proof of the existence results will be divided in several propositions. Before, we need to introduce some definitions. We say that $J_{\lambda, \mu}$ satisfies the $(PS)_{c}^{\ast}$ condition, with respect to $\{Y_{n}\}$, if any sequence $\{u_{n}\}\subset H_{0}^{1}(\Omega)$, such that
\begin{equation}\label{seque}
u_{n}\in Y_{n}, J_{\lambda, \mu}(u_{n})\to c \ \mbox{and} \ (J_{\lambda, \mu}|_{Y_{n}})'(u_{n})\to 0
\end{equation}
contains a subsequence converging to a critical point of $ J_{\lambda, \mu}$. Any sequence $\{u_{n}\}\subset H_{0}^{1}(\Omega)$ satisfying \eqref{seque} is said to be a $(PS)_{c}^{\ast}$ for $J_{\lambda, \mu}$. It is well known that the $(PS)_{c}^{\ast}$ condition implies the classical $(PS)_{c}$ condition, see \cite{Wil}.

\begin{proposition}\label{prop30}
Suppose $(\vartheta_{1})-(\vartheta_3)$ hold. 
\begin{enumerate}
\item[$(i)$] If $p=4$, then $J_{\lambda, \mu}$ satisfies the $(PS)_{c}^{\ast}$ condition, for all $1<q<4$, $\lambda\in\R$ and $\mu<\lambda_{1}\alpha^{2}/4$;
\item[$(ii)$] If $p\neq 4$, then $J_{\lambda, \mu}$ satisfies the $(PS)_{c}^{\ast}$ condition, for all $1<q<\min\{4, p\}$ and $\lambda, \mu\in\R$.
\end{enumerate} 
\end{proposition}

\begin{proof}
$(i)$ Let $p=4$ and $\{u_{n}\}$ be a $(PS)_{c}^{\ast}$ sequence for $J_{\lambda, \mu}$, i.e., \eqref{seque} holds. If $\lambda>0$ and $\mu\leq 0$, it follows by Lemma \ref{prop1}$(v)$ that  
$$
 C+C_{0}\|u_{n}\|\geq J_{\lambda, \mu}(u_{n})-\frac{1}{p}(J_{\lambda, \mu}|_{Y_{n}})'(u_{n})u_{n}\geq\left(\frac{1}{2}-\frac{1}{p}\right)\|u_{n}\|^{2}-\lambda\left(\frac{1}{q}-\frac{1}{2p}\right)\int_{\Omega}\left|f(u_{n})\right|^{q} dx,
$$
Now, we have to consider two cases: if $1<q\leq 2$, we conclude from Lemma \ref{prop1}$(iv)$ and Sobolev embedding that
\begin{equation}\label{AAA}
C+C_{0}\|u_{n}\|\geq\left(\frac{1}{2}-\frac{1}{p}\right)\left\|u_{n}\right\|^{2}-\lambda\left(\frac{1}{q}-\frac{1}{2p}\right)C_{1}\|u_{n}\|^{q}.
\end{equation}

Before consider the case $2<q<4$, observe that, we cannot use the Lemma \ref{prop1}$(iv)$ in the same way as previously because $|u|^{q}$ might not be integrable. To overcome this difficulty, we note that, by items $(v)$ and $(vi)$ of Lemma \ref{prop1} 
\begin{equation}\label{mus}
|f(s)|\leq (8/\alpha^{2})^{1/4}|s|^{1/2},
\end{equation}
for all $s\in\R$. By Lemma \ref{prop1}$(iv)$, for each $2\leq r\leq 22^{\ast}$,
\begin{equation}\label{fini}
\int_{\Omega}\left|f(u)\right|^{r} dx\leq  (8/\alpha^{2})^{r/4}\int_{\Omega}\left|u\right|^{r/2} dx.
\end{equation}

Thus, if $2<q<4$, it follows from \eqref{fini} and Sobolev embedding that
\begin{equation}\label{BBB}
C+C_{0}\|u_{n}\|\geq\left(\frac{1}{2}-\frac{1}{p}\right)\left\|u_{n}\right\|^{2}-\lambda\left(\frac{1}{q}-\frac{1}{2p}\right)(8/\alpha^{2})^{q/4}C_{1}\|u_{n}\|^{q/2}.
\end{equation}
By \eqref{AAA} and \eqref{BBB}, $\{u_{n}\}$ is bounded in $H_{0}^{1}(\Omega)$.  If $\lambda, \mu>0$, by Lemma \ref{prop1}$(v)$, \eqref{fini} and Sobolev embedding, we have
\begin{eqnarray*}
C+C_{0}\|u_{n}\|&\geq& J_{\lambda, \mu}(u_{n})-\frac{1}{4}(J_{\lambda, \mu}|_{Y_{n}})'(u_{n})u_{n}\\
&\geq& \left(\frac{1}{4}-\frac{\mu}{\lambda_{1}\alpha^{2}}\right)\|u\|^{2}-\lambda\left(\frac{1}{q}-\frac{1}{8}\right)\int_{\Omega}\left|f(u_{n})\right|^{q} dx.
\end{eqnarray*}
Hence $\{u_{n}\}$ is bounded in $H_{0}^{1}(\Omega)$, if $\mu<\lambda_{1}\alpha^{2}/4$.

\medskip

On the other hand, if $\lambda, \mu\leq 0$ we get
$$
 C+C_{0}\|u_{n}\|\geq J_{\lambda, \mu}(u_{n})-\frac{1}{p}(J_{\lambda, \mu}|_{Y_{n}})'(u_{n})u_{n}\geq\left(\frac{1}{2}-\frac{1}{p}\right)\|u_{n}\|^{2}-\lambda\left(\frac{1}{q}-\frac{1}{p}\right)\int_{\Omega}\left|f(u_{n})\right|^{q} dx,
$$
showing that $\{u_{n}\}$ is bounded in $H_{0}^{1}(\Omega)$. If $\lambda\leq 0$ and $\mu>0$,
\begin{eqnarray*}
C+C_{0}\|u_{n}\|&\geq& J_{\lambda, \mu}(u_{n})-\frac{1}{4}(J_{\lambda, \mu}|_{Y_{n}})'(u_{n})u_{n}\\
&\geq& \left[\frac{1}{4}-\frac{\mu}{\lambda_{1}\alpha^{2}}\right]\|u\|^{2}-\lambda\left(\frac{1}{q}-\frac{1}{4}\right)\int_{\Omega}\left|f(u_{n})\right|^{q} dx.
\end{eqnarray*}
Therefore $\{u_{n}\}$ is again bounded in $H_{0}^{1}(\Omega)$, if $\mu<\lambda_{1}\alpha^{2}/4$. Thence, up to a subsequence, we have
\begin{equation}\label{weak}
u_{n}\rightharpoonup u \ \mbox{in} \ H_{0}^{1}(\Omega),
\end{equation}
\begin{equation}\label{eqcap31}
\int_{\Omega}f'(u_{n})\left|f(u_{n})\right|^{q-2}f(u_{n})(u_{n}-u) dx\to 0
\end{equation}
and
\begin{equation}\label{eqcap32}
\int_{\Omega}f'(u_{n})\left|f(u_{n})\right|^{p-2}f(u_{n})(u_{n}-u) dx\to 0.
\end{equation}

Defining $v_{n}:=P_{Y_{n}}u$ as been the orthogonal projection of $u$ onto $Y_{n}$, we have 
\begin{equation}\label{coffee}
v_{n}\to u \ \mbox{in} \ H_{0}^{1}(\Omega).
\end{equation}
Since $u_{n}-v_{n}\in Y_{n}$ and $\{u_{n}-v_{n}\}$ is bounded in $H_{0}^{1}(\Omega)$, we conclude that
$$
(J_{\lambda, \mu}|_{Y_{n}})'(u_{n})(u_{n}-v_{n})=o_{n}(1).
$$
Thence,
\begin{eqnarray*}
&&\int_{\Omega}\nabla u_{n}\nabla (u_{n}-v_{n})=\\
&&\lambda\int_{\Omega}f'(u_{n})\left|f(u_{n})\right|^{q-2}f(u_{n})(u_{n}-v_{n}) dx+\mu\int_{\Omega}f'(u_{n})\left|f(u_{n})\right|^{p-2}f(u_{n})(u_{n}-v_{n}) dx+o_{n}(1).
\end{eqnarray*}
By \eqref{weak}, (\ref{eqcap31}), (\ref{eqcap32}) and \eqref{coffee}, we conclude that
\begin{equation}\label{ulti}
\|u_{n}\|^{2}=\|v_{n}\|^{2}+o_{n}(1).
\end{equation}
The result follows now from \eqref{weak} and \eqref{coffee}.

\medskip

$(ii)$ Let $p\neq 4$ and $\{u_{n}\}$ be a $(PS)_{c}^{\ast}$ sequence for $J_{\lambda, \mu}$. If $\lambda>0$ and $\mu\leq 0$ we can reason exactly like in the case $p=4$. On the other hand, if $\lambda, \mu>0$ we have to consider separately two cases: if $p<4$, it follows by Lemma \ref{prop1}$(v)$, \eqref{fini} and Sobolev embedding that  
\begin{eqnarray*}
 C+C_{0}\|u_{n}\|&\geq&J_{\lambda, \mu}(u_{n})-\frac{1}{p}(J_{\lambda, \mu}|_{Y_{n}})'(u_{n})u_{n}\\
 &\geq&\left(\frac{1}{2}-\frac{1}{p}\right)\|u_{n}\|^{2}-\frac{\mu}{2p}(8/\alpha^{2})^{p/4}C_{1}\|u_{n}\|^{p/2}-\lambda\left(\frac{1}{q}-\frac{1}{2p}\right)\int_{\Omega}\left|f(u_{n})\right|^{q} dx.
\end{eqnarray*}
By estimating the last installment as \eqref{AAA} and \eqref{BBB} we conclude that $\{u_{n}\}$ is bounded in $H_{0}^{1}(\Omega)$. In the case $p> 4$, it is sufficient to note that, by Lemma \ref{prop1}$(v)$ 
$$
 C+C_{0}\|u_{n}\|\geq J_{\lambda, \mu}(u_{n})-\frac{2}{p}(J_{\lambda, \mu}|_{Y_{n}})'(u_{n})u_{n}\geq\left(\frac{1}{2}-\frac{2}{p}\right)\|u_{n}\|^{2}-\lambda\left(\frac{1}{q}-\frac{1}{p}\right)\int_{\Omega}\left|f(u_{n})\right|^{q} dx.
$$
Once more time the boundedness of $\{u_{n}\}$ in $H_{0}^{1}(\Omega)$ follows from a reasoning similar to \eqref{AAA} and \eqref{BBB}. 

\medskip

Finally, if $\lambda, \mu\leq 0$, we argue exactly like in the case $p=4$ and, if $\lambda\leq 0$ and $\mu>0$, we have
\begin{eqnarray*}
 C+C_{0}\|u_{n}\|&=&J_{\lambda, \mu}(u_{n})-\frac{1}{p}(J_{\lambda, \mu}|_{Y_{n}})'(u_{n})u_{n}\\
 &\geq&\left(\frac{1}{2}-\frac{1}{p}\right)\|u_{n}\|^{2}-\frac{\mu}{2p}(8/\alpha^{2})^{p/4}C_{1}\|u_{n}\|^{p/2}-\lambda\left(\frac{1}{q}-\frac{1}{p}\right)\int_{\Omega}\left|f(u_{n})\right|^{q} dx,
\end{eqnarray*}
when $p<4$, and 
$$
 C+C_{0}\|u_{n}\|\geq J_{\lambda, \mu}(u_{n})-\frac{2}{p}(J_{\lambda, \mu}|_{Y_{n}})'(u_{n})u_{n}\geq\left(\frac{1}{2}-\frac{2}{p}\right)\|u_{n}\|^{2}-\lambda\left(\frac{1}{q}-\frac{2}{p}\right)\int_{\Omega}\left|f(u_{n})\right|^{q} dx,
$$
when $p>4$. In all cases we can conclude that $\{u_{n}\}$ is bounded in $H_{0}^{1}(\Omega)$. Now the result follows exactly equal to the case $p=4$.

\medskip

\end{proof}

\begin{proposition}\label{prop10} 

Suppose $(\vartheta_{1})-(\vartheta_3)$, $4<p<22^{\ast}$ and $\mu>0$. Then there exist $0<r_{k}<\rho_{k}$ such that:
\begin{equation}\label{eti1}
\max_{u\in Y_{k}, \|u\|=\rho_{k}}J_{\lambda, \mu}(u)\leq 0.
\end{equation}
and 
\begin{equation}\label{eti2}
\inf_{u\in Z_{k}, \left\|u\right\|=r_{k}}J_{\lambda, \mu}(u)\to \infty \ \mbox{as} \ k\to\infty.
\end{equation}
\end{proposition}

\begin{proof}

To prove \eqref{eti1}, observe that by Lemma \ref{prop1}$(v)$
$$
|f(s)|\geq f(1)|s|^{1/2}, \ \mbox{if $|s|>1$}.
$$
Thus, for each $u\in S_{k}$ and $\rho>0$
$$
J_{\lambda, \mu}(\rho u)\leq\frac{1}{2}\rho^{2}+\frac{|\lambda|}{q}\int_{\Omega}|f(\rho u)|^{q}dx-\frac{\mu}{p}f(1)^{p}\rho^{p/2}\int_{[1<|\rho u|]}\left| u\right|^{p/2}dx.
$$
By Lemma \ref{lema2}$(ii)$, there exist positive constants $\alpha_{k}, \beta_{k}(p/2)$ such that, for every $u\in S_{k}$ and $\rho>\alpha_{k}$, we get
\begin{equation}\label{ineq1}
J_{\lambda, \mu}(\rho u)\leq\frac{1}{2}\rho^{2}+\frac{|\lambda|}{q}\int_{\Omega}|f(\rho u)|^{q}dx-\frac{\mu}{p}f(1)^{p}\beta_{k}(p/2)\rho^{p/2}.
\end{equation}
Now, we are going to consider two cases: If $1<q\leq 2$, it follows from Lemma \ref{prop1}$(iv)$ and Sobolev embedding that
$$
J_{\lambda, \mu}(\rho u)\leq\frac{1}{2}\rho^{2}+\frac{|\lambda|}{q}C_{1}\rho^{q}-\frac{\mu}{p}f(1)^{p}\beta_{k}(p/2)\rho^{p/2}.
$$
Since $p>4$, choosing $\rho_{k}>\max\{1, [p(1/2+|\lambda|C_{1}/q)/\mu f(1)^{p}\beta_{k}(p/2)]^{2/(p-4)}\}$, we have
$$
J_{\lambda, \mu}(\rho_{k} u)\leq\left(\frac{1}{2}+\frac{|\lambda|}{q}C_{1}\right)\rho_{k}^{2}-\frac{\mu}{p}f(1)^{p}\beta_{k}(p/2)\rho_{k}^{p/2}<0,
$$
for all $u\in S_{k}$. On the hand, if $2<q<4$, by \eqref{ineq1}, \eqref{fini} and Sobolev embedding, we have
$$
J_{\lambda, \mu}(\rho u)\leq\frac{1}{2}\rho^{2}+\frac{|\lambda|}{q}(8/\alpha^{2})^{q/4}C_{1}\rho^{q/2}-\frac{\mu}{p}f(1)^{p}\beta_{k}(p/2)\rho^{p/2}.
$$
Therefore, choosing $\rho_{k}>\max\{1, [p(1/2+|\lambda|(8/\alpha^{2})^{q/4}C_{1}/q)/\mu f(1)^{p}\beta_{k}(p/2)]^{2/(p-4)}\}$, we have
$$
J_{\lambda, \mu}(\rho_{k} u)\leq\left[\frac{1}{2}+\frac{|\lambda|}{q}(8/\alpha^{2})^{q/4}C_{1}\right]\rho_{k}^{2}-\frac{\mu}{p}f(1)^{p}\beta_{k}(p/2)\rho_{k}^{p/2}<0,
$$
for all $u\in S_{k}$. This proves \eqref{eti1}.

\medskip


To prove \eqref{eti2}, note that for any $1\leq r< 2^{\ast}$, we can define
\begin{equation}\label{sup1}
\theta_{r, k}:=\sup_{u\in Z_{k}\backslash\{0\}}\frac{|u|_{r}}{\|u\|}.
\end{equation}
It is a straightforward consequence of compact Sobolev embeddings that 
\begin{equation}\label{coisa}
\theta_{r, k}\to 0 \ \mbox{as $k\to\infty$},
\end{equation}
see Lemma 3.8 in \cite{Wil}.
If $1<q< 2$, by Lemma \ref{prop1}$(iv)$ and \eqref{fini}
$$
J_{\lambda, \mu}(u)\geq \frac{1}{2}\|u\|^{2}-\frac{|\lambda|}{q}\int_{\Omega}\left|u\right|^{q} dx-\frac{\mu}{p}(8/\alpha^{2})^{p/4}\int_{\Omega}\left|u\right|^{p/2} dx,
$$
By Sobolev embeddings and \eqref{sup1},
$$
J_{\lambda, \mu}(u)\geq \frac{1}{2}\|u\|^{2}-\frac{|\lambda|}{q}C_{1}\|u\|^{q}-\frac{\mu}{p}(8/\alpha^{2})^{p/4}\theta_{p/2, k}^{p/2}\|u\|^{p/2},
$$
for all $u\in Z_{k}$. Since $1<q<2$, for $\left\|u\right\|\geqslant R_{\ast}$ with $R_{\ast}>0$ large enough,
$$
\frac{|\lambda|}{q}C_{1}\left\|u\right\|^{q}<\frac{1}{r}\left\|u\right\|^{2},
$$
for some $r>2p/(p-2)$. Thus, for $\left\|u\right\|\geqslant R_{\ast}$, we get
\begin{equation}\label{achim}
J_{\lambda, \mu}(u)\geqslant\left(\frac{1}{2}-\frac{1}{r}\right)\left\|u\right\|^{2}-\frac{\mu}{p}(8/\alpha^{2})^{p/4}\theta_{p/2, k}^{p/2}\|u\|^{p/2}.
\end{equation}
It follows from \eqref{coisa} that, by choosing $r_{k}=1/[\mu (8/\alpha^{2})^{p/4}\theta_{p/2, k}^{p/2}]^{2/(p-4)}$, there exists $k_{0}\in\N$ such that $r_{k}\geqslant R_{\ast}$ for all $k\geqslant k_{0}$. Therefore, 
\begin{equation}\label{eq3153}
J_{\lambda, \mu}(u)\geqslant\left(\frac{r-2}{2r}-\frac{1}{p}\right)r_{k}^{2},
\end{equation}
for all $u\in Z_{k}$ with $\left\|u\right\|=r_{k}$ and $k\geqslant k_{0}$. Since $r_{k}\to\infty$ as $k\to\infty$, the result follows. If $2\leq q<4$, it follows from \eqref{fini} that
$$
J_{\lambda, \mu}(u)\geq \frac{1}{2}\|u\|^{2}-\frac{|\lambda|}{q}(8/\alpha^{2})^{q/4}\int_{\Omega}\left|u\right|^{q/2} dx-\frac{\mu}{p}(8/\alpha^{2})^{p/4}\int_{\Omega}\left|u\right|^{p/2} dx,
$$
By Sobolev embeddings and \eqref{sup1},
$$
J_{\lambda, \mu}(u)\geq \frac{1}{2}\|u\|^{2}-\frac{|\lambda|}{q}(8/\alpha^{2})^{q/4}C_{1}\|u\|^{q/2}-\frac{\mu}{p}(8/\alpha^{2})^{p/4}\theta_{p/2, k}^{p/2}\|u\|^{p/2},
$$
Now, since $1\leq q/2<2$, we can proceed in an analogous way to the case $1<q<2$ for the choice of $r_{k}$. Since we can choose $\rho_{k}$ even greater, in order to have $\rho_{k}>r_{k}$, the result follows.


\end{proof}

\begin{proposition}\label{prop20} 
Suppose that $\vartheta$ satisfies $(\vartheta_{1})-(\vartheta_{3})$, $1<q<2$ and $\lambda>0$ hold. Then, there exists $0<r_{k}<\rho_{k}$ such
\begin{enumerate}
\item[$(i)$] $\inf_{u\in Z_{k}, \left\|u\right\|=\rho_{k}}J_{\lambda, \mu}(u)\geq 0$;

\item[$(ii)$] $\max_{u\in Y_{k}, \|u\|=r_{k}}J_{\lambda, \mu}(u)< 0$;

\item[$(iii)$] $\inf_{u\in Z_{k}, \|u\|\leq \rho_{k}}J_{\lambda, \mu}(u)\to0$ as $k\to\infty$.
\end{enumerate}
\end{proposition}

\begin{proof} 

$(i)$  Let us consider $p\geq 4$. Since $1<q<2$, by Lemma \ref{prop1}$(iv)$, \eqref{fini} and \eqref{sup1}, we get
\begin{eqnarray}\nonumber\label{minino}
J_{\lambda, \mu}(u)&\geq &\frac{1}{2}\|u\|^{2}-\frac{\lambda}{q}\int_{\Omega}\left|u\right|^{q}dx-\frac{|\mu|}{p}(8/\alpha^{2})^{p/4}\int_{\Omega}\left|u\right|^{p/2}dx\\
&\geq&\frac{1}{2}\|u\|^{2}-\frac{\lambda}{q}\theta_{q, k}^{q}\|u\|^{q}-\frac{|\mu|}{p}(8/\alpha^{2})^{p/4}\theta_{p/2, k}^{p/2}\|u\|^{p/2}, 
\end{eqnarray}
for all $u\in Z_{k}$. If $p\geq 4$, there exists $\delta>0$ small enough, such that 
\begin{equation}\label{eqcc3}
\frac{|\mu|}{p}(8/\alpha^{2})^{p/4}\theta_{p/2, k}^{p/2}\left\|u\right\|^{p/2}\leqslant\frac{1}{4}\left\|u\right\|^{2}, 
\end{equation}
for all $u\in Z_{k}$ with $\|u\|\leq \delta$ (and $k$ large enough if $p=4$). Thus, by choosing 
$$
\rho_{k}=(4\lambda\theta_{q, k}^{q}/q)^{1/(2-q)},
$$ 
we have $(1/4)\rho_{k}^{2}=(\lambda/q)\theta_{q, k}^{q}\rho_{k}^{q}$. Consequently, $\rho_{k}\to0$ as $k\to\infty$ and, therefore, there exists $k_{0}>0$ satisfying $\rho_{k}\leq \delta$ for all $k\geq k_{0}$. Finally, by (\ref{eqcc3})
\begin{equation}\label{eqccx4}
J_{\lambda, \mu}(u)\geq\frac{1}{4}\left\|u\right\|^{2}-\frac{\lambda}{q}\theta_{q, k}^{q}\left\|u\right\|^{q}=0
\end{equation}
for all $u\in Z_{k}$, $k\geqslant k_{0}$, with $\left\|u\right\|=\rho_{k}$. On the other hand, if $2<p<4$, we conclude from \eqref{minino} that
\begin{equation}\label{minina}
J_{\lambda, \mu}(u)\geq \frac{1}{2}\|u\|^{2}-\left[\frac{\lambda}{q}+\frac{|\mu|}{p}(8/\alpha^{2})^{p/4}\right]\eta_{k}^{\gamma}\|u\|^{\gamma},
\end{equation}
for all $u\in Z_{k}$ with $\|u\|<1$, $1<\gamma:=\min\{q, p/2\}<2$, $\eta_{k}:=\max\{\theta_{q, k}, \theta_{p/2, k}\}$ and $k\geq k_{0}$. Thus, by choosing
$$
\rho_{k}=\left\{2[\lambda/q+|\mu|(8/\alpha^{2})^{p/4}/p]\eta_{k}^{\gamma}\right\}^{1/(2-\gamma)},
$$
with $k\geq k_{0}$, the result follows.

\medskip

$(ii)$  By Lemma \ref{prop1}$(iii)$, there exists $\varepsilon>0$ such that
$$
|f(s)|\geq \varepsilon|s|,
$$
for all $|s|\leq 1$. Thus, 
$$
J_{\lambda, \mu}(u)\leq \frac{1}{2}\left\|u\right\|^{2}-\frac{\lambda}{q}\varepsilon^{q}\int_{[|u|\leq 1]}\left|u\right|^{q} dx+\frac{|\mu|}{p}\int_{\Omega}\left|f(u)\right|^{p}dx.
$$
By the second part of Lemma \ref{lema2}$(i)$ and Lemma \ref{prop1}$(iv)$, we have
$$
J_{\lambda, \mu}(r u)\leq \frac{1}{2}r^{2}-\frac{\lambda}{q}\varepsilon^{q}\int_{\Omega}\left|r u\right|^{q} dx+\frac{|\mu|}{p}\int_{\Omega}\left|r u\right|^{p}dx,
$$
for all $u\in S_{k}$ and $0<r<\tau_{k}$. Since $Y_{k}$ has finite dimension, there exists $\zeta_{k}(q)>0$ such that
$$
J_{\lambda, \mu}(r u)\leq \frac{1}{2}r^{2}-\frac{\lambda}{q}\varepsilon^{q}\zeta_{k}(q)r^{q}+\frac{|\mu|}{p}\int_{\Omega}\left|r u\right|^{2}dx,
$$
for all $u\in S_{k}$ and $0<r<\tau_{k}$, where in the last installment we use the fact that $p>2$. By Sobolev embeddings
$$
J_{\lambda, \mu}(r u)\leq \frac{1}{2}r^{2}-\frac{\lambda}{q}\varepsilon^{q}\zeta_{k}(q)r^{q}+\frac{|\mu|}{p}C_{1}r^{2}.
$$
Thence,
$$
J_{\lambda, \mu}(r u)\leq \left(\frac{1}{2}+\frac{|\mu|}{p}C_{1}\right)r^{2}-\frac{\lambda}{q}\varepsilon^{q}\zeta_{k}(q)r^{q},
$$
for all $0<r<\min\{1, \rho_{k}, \tau_{k}\}$. Since $1<q<2$, by choosing 
$$
0<r_{k}<\min\{1, \tau_{k}, \rho_{k}, [\lambda \varepsilon^{q}\zeta_{k}(q)/q(1/2+|\mu|C_{1}/p)]^{1/(2-q)}\},
$$ 
the item is proven. 

%

\medskip

$(iii)$ By \eqref{eqccx4} and \eqref{minina}, we conclude that
$$
o_{k}(1)\leq b_{k}:=\inf_{u\in Z_{k}, \left\|u\right\|\leq \rho_{k}}J_{\lambda, \mu}(u)\leq J_{\lambda, \mu}(0)=0,
$$
where, $o_{k}(1)\to0$ as $k\to\infty$. Consequently, $b_{k}\to 0$ as $k\to\infty$.
\end{proof}

$\square$

\medskip

{\bf Proof of Theorem \ref{teor3}$(i)$:}

\medskip

Since $J_{\lambda, \mu}$ is an even functional, the first part of Theorem \ref{teor3}$(i)$ is a direct consequence of Fountain Theorem in \cite{Wil} and Propositions \ref{prop30}$(ii)$ and \ref{prop10}. To prove the second part, observe that if $1<q<2$, it follows from $\mu>0$, Lemma \ref{prop1}$(iv)$, \eqref{fini} and Sobolev embeddings, that
$$
J_{\lambda, \mu}(u)\geq \frac{1}{2}\|u\|^{2}-\frac{|\lambda|}{q}C_{1}\|u\|^{q}-\frac{\mu}{p}(8/\alpha^{2})^{p/4}\theta_{p/2, m}^{p/2}\|u\|^{p/2},
$$
for all $u\in Z_{m}$. On the other hand, if $2\leq q<4$, it follows from $\mu>0$, \eqref{fini} and Sobolev embeddings, that
$$
J_{\lambda, \mu}(u)\geq \frac{1}{2}\|u\|^{2}-\frac{|\lambda|}{q}(8/\alpha^{2})^{q/4}C_{2}\|u\|^{q/2}-\frac{\mu}{p}(8/\alpha^{2})^{p/4}\theta_{p/2, m}^{p/2}\|u\|^{p/2},
$$
for all $u\in Z_{m}$. Consequently,
$$
J_{\lambda, \mu}(u)\geq \frac{1}{2}\|u\|^{2}-\frac{|\lambda|}{q}C_{3}\|u\|^{\alpha(q)}-\frac{\mu}{p}(8/\alpha^{2})^{p/4}\theta_{p/2, m}^{p/2}\|u\|^{p/2},
$$
where $\alpha:(1, 4)\to [1, 2)$ is give by $\alpha(s)=s$ if $1<s<2$ and $\alpha(s)=s/2$ if $2\leq s<4$. Thence, there exists $R_{\ast}$ large enough such that 
$$
\frac{1}{4}\|u\|^{2}\geq \frac{|\lambda|}{q}C_{3}\|u\|^{\alpha(q)},
$$
for all $u\in Z_{m}$ with $\|u\|\geq R_{\ast}$. Since $p<4$,
$$
J_{\lambda, \mu}(u)\geq \left[\frac{1}{4}-\frac{\mu}{p}(8/\alpha^{2})^{p/4}\theta_{p/2, m}^{p/2}\right]\|u\|^{p/2},
$$
for all $u\in Z_{m}$ with $\|u\|\geq \max\{R_{\ast}, 1\}$. Observe that there exists $m_{0}>0$ such that 
$$
\frac{1}{4}>\frac{\mu}{p}(8/\alpha^{2})^{p/4}\theta_{p/2, m}^{p/2},
$$
for all $m\geq m_{0}$. By choosing $r_{m}=\max\{R_{\ast}, m\}$, we have
\begin{equation}\label{explo}
\inf_{u\in Z_{m}, \|u\|=r_{m}}J_{\lambda, \mu}(u)\to\infty \ \mbox{as} \ m\to\infty.
\end{equation} 
Finally, by items $(iv)$ and $(v)$ of Lemma \ref{prop1} and \eqref{fini}, there exists $C>0$ such that
$$
J_{\lambda, \mu}(\rho u)\leq \frac{\rho^{2}}{2}+\frac{|\lambda|}{q}C \rho^{\alpha(q)}\int_{\Omega}|u|^{\alpha(q)}dx-\frac{\mu}{p}f(1)^{p}\rho^{p/2}\int_{[|\rho u|>1]}|u|^{p/2}dx,
$$
for all $u\in S_{m}$. It follows from Lemma \ref{lema2}$(ii)$ and Sobolev embedding that there exists $\alpha_{m}, \beta_{m}(p/2)>0$ such that
$$
J_{\lambda, \mu}(\rho_{m} u)\leq \frac{\rho_{m}^{2}}{2}+\frac{|\lambda|}{q}C_{1} \rho_{m}^{\alpha(q)}-\frac{\mu}{p}f(1)^{p}\beta_{m}(p/2)\rho_{m}^{p/2},
$$
for some $\rho_{m}>\max\{\alpha_{m}, r_{m}\}$ and for all $u\in S_{m}$. Therefore, there exists $\mu_{m}>0$ such that
\begin{equation}\label{tal}
\max_{u\in Y_{m}, \|u\|=\rho_{m}}J_{\lambda, \mu}(u)\leq 0,
\end{equation}
for all $\mu>\mu_{m}$. To finish the proof, let us define
$$
 B_{m}=\{u\in Y_{m}: \|u\|\leq \rho_{m}\},
$$ 
$$
\Gamma_{m}=\{\gamma\in C(B_{m}, H_{0}^{1}(\Omega)): \gamma \ \mbox{is odd and} \ \gamma_{|_{\partial B_{m}}}=id\}
$$
and
$$
c_{m}=\inf_{\gamma\in \Gamma_{m}}\max_{u\in B_{m}}J_{\lambda, \mu}(\gamma(u)).
$$
By definition of $c_{m}$ and Lemma 3.4 in \cite{Wil}, we have
\begin{equation}\label{tico}
\infty>c_{m}\geq \inf_{u\in Z_{m}, \|u\|=r_{m}}J_{\lambda, \mu}(u),
\end{equation}
for all $m$. On the other hand, by \eqref{explo}, we conclude that 
$$
\inf_{u\in Z_{m}, \|u\|=r_{m}}J_{\lambda, \mu}(u)>0,
$$
for all $m\geq m_{0}$. It is also a consequence of \eqref{explo} and \eqref{tico} that given $k\in\N$, there exists $m(k)>m_{0}$ with $k\leq m(k)-m_{0}$, such that we have at least $k$ different numbers $c_{j}$ when $m_{0}\leq j\leq m(k)$. Thus, by \eqref{tal} and Theorem 3.5 in \cite{Wil}, there exist $\mu_{k}:=\mu_{m(k)}>0$ and a $(PS)_{c_{j}}$-sequence for $J_{\lambda, \mu}$, for each $m_{0}\leq j\leq m(k)$, whenever $\mu>\mu_{k}$. Finally, by Proposition \ref{prop30}$(ii)$, follows that the numbers $c_{j}$ are critical points of $J_{\lambda, \mu}$ as $\mu>\mu_{k}$.
$\square$

\medskip

{\bf Proof of Theorem \ref{teor3}$(ii)$:}

\medskip

Since $J_{\lambda, \mu}$ is an even functional, the proof of first part of Theorem \ref{teor3}$(ii)$ follows from Dual Fountain Theorem in \cite{Wil} and Propositions \ref{prop30}$(ii)$ and \ref{prop20}. To prove the second part, note that, since $2\leq q<4$ and $\lambda>0$, it follows by \eqref{fini} and Sobolev embeddings, that
$$
J_{\lambda, \mu}(u)\geq \frac{1}{2}\|u\|^{2}-\frac{\lambda}{q}(8/\alpha^{2})^{q/4}\theta_{q/2, m}^{q/2}\|u\|^{q/2}-\frac{|\mu|}{p}(8/\alpha^{2})^{p/4}\theta_{p/2, m}^{p/2}\|u\|^{p/2},
$$
for all $u\in Z_{m}$. Thus, for $m$ large enough, we have $0<\eta_{m}:=\max\{\theta_{q/2, m}, \theta_{p/2, m}\}<1$ and 
\begin{equation}\label{iei}
J_{\lambda, \mu}(u)\geq \frac{1}{2}\|u\|^{2}-\left(\frac{\lambda}{q}+\frac{|\mu|}{p}\right)(8/\alpha^{2})^{q/4}\eta_{m}^{q/2}\|u\|^{q/2},
\end{equation}
for all $u\in Z_{m}$ with $\|u\|<1$. By choosing $\rho_{m}=\left[2\left(\lambda/q+|\mu|/p\right)(8/\alpha^{2})^{q/4}\eta_{m}^{q/2}\right]^{2/(4-q)}$, it follows that for $m\geq m_{\ast}$, with $m_{\ast}$ large enough
\begin{equation}\label{posis}
\inf_{u\in Z_{m}, \|u\|=\rho_{m}}J_{\lambda, \mu}(u)\geq 0.
\end{equation}

\medskip

On the other hand, by Lemma \ref{prop1}$(iii)$ and \eqref{fini}
$$
J_{\lambda, \mu}(r u)\leq \frac{r^{2}}{2}-\frac{\lambda}{q}\varepsilon^{q} \int_{[|r u|\leq 1]}|ru|^{q}dx+\frac{|\mu|}{p}(8/\alpha^{2})^{p/4}r^{p/2}\int_{\Omega}|u|^{p/2}dx,
$$
for all $u\in S_{m}$. It follows from Lemma \ref{lema2}$(i)$ that there exists $\tau_{m}>0$ such that
$$
J_{\lambda, \mu}(r_{m} u)\leq \frac{r_{m}^{2}}{2}-\frac{\lambda}{q}\varepsilon^{q}r_{m}^{q}\int_{\Omega}|u|^{q}dx+\frac{|\mu|}{p}(8/\alpha^{2})^{p/4}r_{m}^{p/2}\int_{\Omega}|u|^{p/2}dx,
$$
for some $0<r_{m}<\min\{\tau_{m}, \rho_{m}\}$ fixed and for all $u\in S_{m}$. Despite $q$ can be greater than $2^{\ast}$ when the dimension $N$ is large enough, it is a consequence of definition of $Y_{m}$ that $Y_{m}\subset L^{\infty}(\Omega)$ and, therefore, $|.|_{q}$ defines a norm in $Y_{m}$. Since $Y_{m}$ has finite dimension,
$$
J_{\lambda, \mu}(r_{m} u)\leq \frac{r_{m}^{2}}{2}-\frac{\lambda}{q}\varepsilon^{q}r_{m}^{q}\zeta_{m}(q)+\frac{|\mu|}{p}(8/\alpha^{2})^{p/4}C_{1}r_{m}^{p/2},
$$
for some $\zeta_{m}(q)>0$. Therefore, there exists $\lambda_{m}>0$ such that
\begin{equation}\label{tals}
b_{m}:=\max_{u\in Y_{m}, \|u\|=r_{m}}J_{\lambda, \mu}(u)< 0,
\end{equation} 
for all $\lambda>\lambda_m$.

\medskip

Finally, by \eqref{iei}, we conclude that
$$
o_{m}(1)\leq \inf_{u\in Z_{m}, \left\|u\right\|\leq \rho_{m}}J_{\lambda, \mu}(u)\leq J_{\lambda, \mu}(0)=0,
$$
where, $o_{m}(1)\to0$ as $m\to\infty$. Showing that 
\begin{equation}\label{converg}
d_{m}:=\inf_{u\in Z_{m}, \left\|u\right\|\leq \rho_{m}}J_{\lambda, \mu}(u)\to 0 \ \mbox{as} \ m\to\infty.
\end{equation}

To finish the proof, for each $t\geq m\geq m_{\ast}$, we are going to apply the Theorem 3.5 in \cite{Wil} to the functional $-J_{\lambda, \mu}$ on $Y_{t}$, for this, let us define:
 $$
 Z_{m}^{t}=\oplus_{j=m}^{t} X_{j},
 $$ 
 $$
 B_{m}^{t}=\{u\in Z_{m}^{t}: \|u\|\leq \rho_{m}\},
 $$ 
 $$
 \Gamma_{m}^{t}=\{\gamma\in C(B_{m}^{t}, Y_{m}): \gamma \ \mbox{is odd and} \ \gamma_{|_{\partial B_{m}^{t}}}=id\}
 $$
 and
 $$
 c_{m}^{t}=\sup_{\gamma\in \Gamma_{m}^{t}}\min_{u\in B_{m}^{t}}J_{\lambda, \mu}(\gamma(u)).
 $$
By definition of $c_{m}^{t}$ and Lemma 3.4 in \cite{Wil}, we have
\begin{equation}\label{nivel1}
d_{m}<c_{m}^{t}\leq b_{m},
\end{equation} 
for all $t\geq m\geq m_{\ast}$. Therefore, up to a subsequence, there exists 
\begin{equation}\label{defesa}
c_{m}\in [d_{m}, b_{m}]
\end{equation} 
such that 
\begin{equation}\label{conte}
c_{m}^{t}\to c_{m} \ \mbox{as} \ t\to\infty.
\end{equation} 
From \eqref{tals}, \eqref{converg} and \eqref{defesa}, given $k\in\N$, there exist $m(k)$ with $k<m(k)-m_{\ast}$ and $\lambda_{k}:=\lambda_{m(k)}>0$ such that we have at least $k$ different numbers $c_{m}$ as $m_{\ast}\leq m\leq m(k)$, whenever $\lambda>\lambda_{k}$. Thus, by Theorem 3.5 in \cite{Wil}, for each $m_{\ast}\leq m\leq m(k)$, there exists $u_{t}\in Y_{t}$ such that
\begin{equation}\label{palais}
c_{m}^{t}-2/t\leq J_{\lambda, \mu}(u_{t})\leq c_{m}^{t}+2/t \ \mbox{and} \ \|(J_{\lambda, \mu}|_{Y_{t}})'(u_{t})\|\leq 8/t,
\end{equation}
whenever $\lambda>\lambda_{k}$. Consequently, by \eqref{conte} and \eqref{palais}, up to a subsequence, $\{u_{t}\}$ is a $(PS)^{\ast}_{c_{m}}$ sequence. By Proposition \ref{prop30}$(ii)$, $c_{m}$ is a critical point of $J_{\lambda, \mu}$ for all $m_{\ast}\leq m\leq m(k)$. The result follows.
$\square$

\medskip

{\bf Proof of Theorem \ref{teor3}$(iii)$:} It is sufficient to argue exactly like in the proof of the second part of Theorem \ref{teor3}$(ii)$ and use Proposition \ref{prop30}$(i)$ instead of Proposition \ref{prop30}$(ii)$.

\medskip

$\square$


\end{document}